\documentclass[a4paper]{amsart}
\usepackage{amsfonts}
\usepackage{amssymb}
\usepackage{graphicx}    
\usepackage{multicol}    

\usepackage{amsmath}
\usepackage{tcolorbox}
\usepackage{mathrsfs}
\usepackage{amsthm}
\usepackage{lastpage}
\usepackage{fancyhdr}
\usepackage{accents}
\usepackage{stmaryrd}
\usepackage{tabularx}
\usepackage{tikz}
\usepackage{tikz-cd}
\usepackage{array}
\usepackage{mathtools}
\usepackage{comment}
\usepackage{subcaption}
\usepackage{adjustbox}
\DeclareFontFamily{OT1}{pzc}{}
\DeclareFontShape{OT1}{pzc}{m}{it}{<-> s * [1.10] pzcmi7t}{}
\DeclareMathAlphabet{\mathpzc}{OT1}{pzc}{m}{it}
\newtheorem{theorem}{Theorem}[section]
\newtheorem*{theorem*}{Theorem}
\newtheorem{lemma}[theorem]{Lemma}
\newtheorem{prop}[theorem]{Proposition}

\newtheorem{eg}[theorem]{Example}

\def\Z{\mathbb{Z}}    

\newtheorem{definition}{Definition}[section]
\def\K{\mathcal{K}}
\def\H{\tilde{H}}

\def\P{\mathcal{P}}

\def\ZZ{\mathcal{Z}}
\def\w{\omega}

\def\Z{\mathbb{Z}}
\def\L{\mathcal{L}}

\def\ker{\text{Ker }}

\newcommand{\gen}[1]{\langle#1\rangle}

\newtheorem{thm}{Theorem}[section]

\theoremstyle{definition}

\theoremstyle{definition}
\newtheorem{construction}[thm]{Construction}

\theoremstyle{definition}

\theoremstyle{remark}

\usepackage{etoolbox}
\makeatletter
\patchcmd{\@maketitle}
  {\ifx\@empty\@dedicatory}
  {\ifx\@empty\@date \else {\vskip3ex \centering\footnotesize\@date\par\vskip1ex}\fi
   \ifx\@empty\@dedicatory}
  {}{}
\patchcmd{\@adminfootnotes}
  {\ifx\@empty\@date\else \@footnotetext{\@setdate}\fi}
  {}{}{}
\makeatother

\newtheorem{innercustomgeneric}{\customgenericname}
\providecommand{\customgenericname}{}
\newcommand{\newcustomtheorem}[2]{%
  \newenvironment{#1}[1]
  {%
   \renewcommand\customgenericname{#2}%
   \renewcommand\theinnercustomgeneric{##1}%
   \innercustomgeneric
  }
  {\endinnercustomgeneric}
}\newcustomtheorem{customprop}{Proposition}
\newcustomtheorem{customthm}{\textbf{Theorem}}

\begin{document}

\title{Double Homology and Wedge-Decomposable Simplicial Complexes}

\author{Donald Stanley, Carlos Gabriel Valenzuela Ruiz}

\date{\today}

\maketitle
\thispagestyle{empty}
\begin{abstract}
    We show a wedge-decomposable simplicial complex has associated double homology $\Z\oplus\Z$ in bidegrees $(0,0)$, $(-1,4)$.
\end{abstract}

\maketitle
\tableofcontents
\section{Introduction}\label{sec:introduction}
    \noindent\textbf{Homology of $\ZZ_\K$:} In toric topology, rather than studying simplicial complexes themselves a finer topological invariant is studied, that is the moment-angle complex $\ZZ_\K$ associated with $\K$. Given that $\K$ is a simplicial complex on $[m]$, $\ZZ_\K$ is a topological space given by a subspace of the polydisk $\left(D^{2}\right)^{m}$ that contains the combinatorial structure of $\K$ (see Section 2). It has a bigraded cell decomposition that induces bigraded homology groups $H_{-k,2l}(\ZZ_\K)$. This is studied extensively in Chapter 4 of \cite{ToricTop}, where they show these groups can be expressed as the sum of the reduced homology of all full subcomplexes of $\K$ (Hochster decomposition, see Theorem 3.1, Equation 3). \\

    \noindent\textbf{Double homology:} A new invariant called the \textit{bigraded double homology of a moment angle complex} was introduced in \cite{docoho}. This was designed to solve a stability problem with the bigraded persistent homology obtained by considering the homology groups of a moment-angle complex \cite{Stab}. Double homology $HH_{*,*}(\ZZ_\K)$ is defined by endowing a cochain complex structure to $H_*(\ZZ_\K)$ and computing its cohomology. There are three ways of obtaining this cochain complex. However, here we shall only consider the approach that uses the so-called \textit{Hochster decomposition}.\\

    \noindent\textbf{Results:} Given simplicial complexes $\K^1,\K^2$ with a common proper face $\sigma$, identifying them along it gives the face sum $\K^1\sqcup_\sigma\K^2$. If a simplicial complex $\K$ can be built in such a way, we say $\K$ is wedge-decomposable (see Section 5). Our main result computes the double homology of such sums (Theorem 5.4). 

\begin{theorem*}
    	If $\K^1$ and $\K^2$ have no ghost vertices then
 \[
	HH_{-k,2l}(\ZZ_{\K^1\sqcup_\sigma\K^2})=\left\{\begin{array}{cl}
		\Z & \textit{ for $(k,l)=(0,0)$ or $(-1,4)$}\\
		0 & \textit{ else}.
	\end{array}\right.\]
\end{theorem*}

 This construction is analogous to that of the clique sum of two graphs. Given two graphs, the clique complex of a clique sum is precisely the face sum of the two clique complexes. Therefore this result gives a homological tool to detect graphs that are not clique sums. Further, Example \ref{counter} showcases a non-wedge-decomposable simplicial complex with associated $HH$ of the same form as the theorem, giving a complete answer to Question 8.2 in \cite{docoho}. A characterization of simplicial complexes with such double homology is unknown. \\

    \noindent\textbf{Organization:}
    \begin{itemize}
        \item Section 2 introduces some terminology.
    
        \item  In Section 3 we review the $HH_*$ construction.
    
        \item In Section 4 we compute the double homology of $\ZZ_\K$ when $\K$ has at least one ghost vertex. To do it we put a filtration on $CH_{*}$ isolating the action of adding the ghost vertex to $\K$.
    
        \item Section 5 gives a proof of the main theorem by obtaining a Mayer--Vietoris--like long exact sequence for $HH_{*}$.
    \end{itemize}

    \noindent\textbf{Acknowledgements:} Part of this work appears in Valenzuela's MSc.~thesis completed at the University of Regina under the supervision of Dr. Donald Stanley who is supported by NSERC RGPIN-05466-2020. We thank the anonymous reviewer for their helpful comments.

\section{Background and notation}

Let $S$ be a totally ordered finite set and let $\K\subseteq \P(S)$ with $\emptyset\in \K$. We say the pair $(\K,S)$ is a \textit{simplicial complex} if for all $\sigma\in \K$ and $\tau\subseteq \sigma$, $\tau\in \K$. We refer to $(\K,S)$ as the simplicial complex $\K$ on $S$, and the latter is referred to as the \textit{vertex set} of $\K$.\\

Elements in a simplicial complex are called \textit{faces}. Given a subset $F\subseteq \P(S)$ we define the \textit{simplicial complex generated by $F$} as the minimal simplicial complex on $S$ which contains $F$   and is denoted by $\gen{F}$.\\

For a simplicial complex $\K$ on $S$, we say $x\in S$ is a ghost vertex if $\{x\}\notin\K$. As we usually do not want to deal with those vertices, we define the \textit{effective vertex set} of $\K$ as the maximal subset of $S$ without ghost vertices of $\K$ and we denote it $V(\K)$.\\

Given a simplicial complex $\K$ on $S$ and $J\subset S$, we define its corresponding \textit{full subcomplex} $\K_J$ as the simplicial complex on $S$ given by $\K_J=\{\sigma\cap J:\sigma\in\K\}$.\\

For $m,n\in \Z$ we will denote $[n,m]:=\{k\in\Z:n\leq k\leq m\}$. For the sake of simplicity, we denote $[m]:=[1,m]$. As all vertex sets we use in this work are finite, we can assume they are of the form $[m]$ for some $m\in\Z$ unless otherwise stated.\\

Simplicial complexes correspond to  the classical \textit{geometric simplicial complexes} through a \textit{geometric realization} functor. Here sets of size $n+1$ correspond to $n$-dimensional simplices. When we talk about the homology of a simplicial complex we mean the simplicial homology of the underlying geometric simplicial complex. For more details on this one can refer to \cite{Hatcher} Section 2.1.\\

One of the most basic tools for computing the homology groups of a space is the Mayer--Vietoris long exact sequence.

\begin{theorem}[{{\cite[\S4.6]{Spanier}}}]
    Given two simplicial complexes $\K^1$, $\K^2$ on $[m]$, there is a natural long exact sequence in homology
    \[\begin{tikzcd}[ampersand replacement=\&]
	\cdots \& {H_{n+1}(\K^1\cup\K^2)} \& {H_n(\K^1\cap\K^2)} \& {H_n(\K^1)\oplus H_n(\K^2)} \& {H_n(\K^1\cup\K^2)} \& \cdots
	\arrow["\partial", from=1-2, to=1-3]
	\arrow["{i_*}", from=1-3, to=1-4]
	\arrow["{j_*}", from=1-4, to=1-5]
	\arrow[from=1-5, to=1-6]
	\arrow[from=1-1, to=1-2]
\end{tikzcd}\]
where $i_*(x)=(\overline{x},-\overline{x})$, and $j_*(a,b)=\overline{a}+\overline{b}$. Similarly, there is an analogous sequence in reduced homology
\[\begin{tikzcd}[ampersand replacement=\&]
	\cdots \& {\H_{n+1}(\K^1\cup\K^2)} \& {\H_n(\K^1\cap\K^2)} \& {\H_n(\K^1)\oplus H_n(\K^2)} \& {\H_n(\K^1\cup\K^2)} \& \cdots
	\arrow["\partial", from=1-2, to=1-3]
	\arrow["{i_*}", from=1-3, to=1-4]
	\arrow["{j_*}", from=1-4, to=1-5]
	\arrow[from=1-5, to=1-6]
	\arrow[from=1-1, to=1-2]
\end{tikzcd}\]
\end{theorem}

\section{A review of the $HH_{*,*}$ construction}

Let $\K$ be a simplicial complex on $[m]$, its associated \textit{moment-angle complex} is given by the polyhedral product \[\ZZ_\K:=\left(D^2,S^1\right)^\K=\bigcup_{\sigma\in\K}\left(D^2,S^1\right)^\sigma\subseteq {\left(D^{2}\right)}^m\]
where $(D^2, S^1)^\sigma=\left(\prod\limits_{i\in \sigma} D^2\right)\times \left(\prod\limits_{i\in [m]\setminus\sigma} S^1\right)$ and the order of the product is given by the order in $[m]$. For more details and examples we refer to \cite{ToricTop}.\\

Given a simplicial complex $\K$ on $[m]$, its corresponding \textit{Stanley-Reisner ring} is given by \[\Z[\K]:=\Z[v_1,\ldots,v_m]/I_\K\] where $I_\K$ is the ideal generated by the monomials $\prod\limits_{i\in J}v_i$ for each $J\in \P([m])\setminus\K$. This is useful as it turns out that this ring describes the cohomology of $\ZZ_\K$ through the following theorem.

\begin{theorem}[{{\cite[\S~4.5]{ToricTop}}}]
    There are isomorphisms of bigraded commutative algebras
    \begin{align}
        H^{*}(\ZZ_\K)&\cong \text{\normalfont Tor }_{\Z[v_1,\ldots,v_m]}(\Z,\Z[\K])\\
                &\cong H(\Lambda[u_1,\ldots,u_m]\otimes \Z[\K],d)\\
                &\cong\bigoplus_{J\subseteq[m]}\H^*(\K_J),
    \end{align}
    where \text{\normalfont{bideg }}$u_i=(-1,2)$, \text{\normalfont{bideg }}$v_i=(0,2)$ and $d$ is given by $d(v_i)=0$ and $d(u_i)=v_i$.
\end{theorem}
The description we care about the most is the last one, often called the \textit{Hochster decomposition}. The name comes from Hochster's theorem \cite{hochster} which describes $\text{Tor}_{\Z[v_1,\ldots,v_m]}(\Z,\Z[\K])$ as a sum of reduced cohomology of full subcomplexes.\\

\noindent The bigrading in (3) is induced by the isomorphism and is given by

\[H^n(\ZZ_\K)\cong \bigoplus_{-k+2l=n}H^{-k,2l}(\ZZ_\K)\hspace{3mm}\text{ where }\hspace{3mm} H^{-k,2l}(\ZZ_\K)\cong\bigoplus_{\substack{J\subseteq [m]\\|J|=l}} \H^{l-k-1}(\K_J). \]

\noindent For homology, there is an analogous description, given by 
\[H_n(\ZZ_\K)\cong \bigoplus_{-k+2l=n}H_{-k,2l}(\ZZ_\K)\hspace{3mm}\text{ where }\hspace{3mm}  H_{-k,2l}(\ZZ_\K)\cong\bigoplus_{\substack{J\subseteq [m]\\|J|=l}} \H_{l-k-1}(\K_J). \]
\noindent Moreover, this description is covariantly functorial with respect to inclusions. Notice that if $l<k$ then automatically $H_{-k,2l}(\ZZ_\K)=0$, so for the rest of this work, when talking about this decomposition we shall assume $l\geq k$.\\

Now we introduce the protagonist of this work. Let $\K$ be a simplicial complex on $[m]$, in \cite{docoho} the homology of $\ZZ_\K$ was endowed with a second differential $d$, giving it a cochain complex structure. The cohomology of this complex was called the double homology of $\ZZ_\K$ and is denoted by $HH_{*,*}(\ZZ_\K)$.\\

\begin{construction} Let $\K$ be a simplicial complex on some totally ordered set $S$. For each $k$, we define the bigraded cochain complex $CH_{*,*}(\ZZ_\K)$ by making
\[CH_{-k,2l}(\ZZ_\K)=\bigoplus_{\substack{J\subseteq S \\|J|=l}}\H_{l-k-1}(\K_J).\]
\noindent To construct the differential, for each $J\subseteq S$ and $p\in \Z$ consider the map $d_{p;J}:\H_p(\K_J)\to \bigoplus_{x\in S\setminus J} \H_p(\K_{J\cup\{x\}})$ given by
\begin{equation}\label{partiald}
    d_{p;J}=(-1)^{p+1}\sum_{x\in S\setminus J}\varepsilon(x,J)\phi_{p;J,x}
\end{equation}
where $\varepsilon(x,J)=(-1)^{|\{j\in J:j<x\}|}$ and $\phi_{p;J,x}:\H_p(\K_J)\to \H_p(\K_{J\cup\{x\}})$ is the map induced by the inclusion~$\K_J\xhookrightarrow{}\K_{J\cup\{x\}}$. Then the differential $d_{-k,2l}:CH_{-k,2l}(\ZZ_\K)\to CH_{-k-1,2l+2}(\ZZ_\K)$ is given by \[d_{-k,2l}=\sum_{\substack{J\subseteq S\\|J|=l}}d_{l-k-1;J}\] Hochster's decomposition induces a bigraded cochain complex structure on $H_{*}(\ZZ_\K)$, taking cohomology yields the \textit{bigraded double homology of $\ZZ_\K$}.\\

Notice that as the differential does not change the homological degree $p$, we can split the cochain complex in a direct sum of cochain complexes as follows

\[CH_{*}(\ZZ_\K)\cong\bigoplus_{n\in \Z} CH_n^*(\ZZ_\K) \hspace{0.5cm}\text{ where }\hspace{0.5cm} CH_{n}^l(\ZZ_\K):=\bigoplus_{\substack{J\subseteq S\\|J|=l}}\H_{n-1}(\K_J).\]
Let $HH_{n}^l(\ZZ_\K):=H^l(CH_{n}^*(\ZZ_\K))$, then $HH_{-k,2l}(\ZZ_\K)\cong HH_{l-k}^l(\ZZ_\K)$. Therefore computing $HH_{*}(\ZZ_\K)$ is equivalent to computing $HH_n^*(\ZZ_\K)$ individually for each integer $n$.\\

\noindent We can analogously construct the bigraded double cohomology of moment angle complexes, for more details on this and examples we refer to \cite{docoho}.
\end{construction}

\section{Computations}
When a simplicial complex has ghost vertices the description of $HH_{*}$ turns out to be extremely simple. We present two computations regarding this:
\begin{prop}\label{boo}
	For any simplicial complex $\K$ with vertex set $[m]$ and set of ghost vertices $G$, \[HH_{-k,0}(\ZZ_\K)\cong \left\{\begin{array}{cl}
		\Z & \text{ for } k=0\\
		0 &\text{ else }	
	\end{array}\right.\] whenever $G=\emptyset$, otherwise $HH_{-k,0}(\ZZ_\K)=0$ for all $k$.
\end{prop}
\begin{proof}
	As $l=0$, the only subset of $[m]$ we run through in the Hochster decomposition is $\emptyset$, therefore \[CH_{-k,0}(\ZZ_\K)=\H_{-k-1}(\K_\emptyset)=\H_{-k-1}(\emptyset)=\left\{\begin{array}{cl}
		\Z & \text{ for } k=0\\
		0 &\text{ else.}	
	\end{array}\right.\]
    Therefore the only relevant case for computing the double homology is $k=0$. Assume $G\neq \emptyset$, it is clear that $\H_{-1}(\K_{\{x\}})\neq 0$ only when $x\in G$, therefore we have that \[CH_{-1,2}(\ZZ_\K)=\bigoplus_{x\in G} \Z\{x\}.\] But by definition, the differential maps $1\in CH_{0,0}(\ZZ_\K)$ to $\sum\limits_{x\in G}g_x\in CH_{-1,2}(\ZZ_\K)$ where $g_x$ is the generator cycle of $\H_{-1}(\K_{\{x\}})$; as this map is injective it follows that $HH_{0,0}(\ZZ_\K)=0$.\\

	On the other hand, if $G=\emptyset$ then for all $x\in [m]$ we have that $\K_{\{x\}}=\{\emptyset, \{x\}\}$ and so $H_{-1}(\K_{\{x\}})=0$ for all $x\in [m]$, therefore $H_{-1,2}(\ZZ_\K)=0$, letting us conclude that $HH_{0,0}(\ZZ_\K)\cong \ker{ (\Z\to 0)}\cong\Z$.
\end{proof}
\begin{theorem}\label{king boo}
	Let $\K$ be a simplicial complex on $[m]$. Then $HH_*(\ZZ_\K)=0$ if and only if $\K$ has at least one ghost vertex.
\end{theorem}
\begin{proof}
    From Proposition \ref{boo} we get that if $\K$ has no ghost vertices then $HH_*(\ZZ_\K)\neq 0$. Relabeling the vertex set we can consider $\K$ to be a simplicial complex on $[0,m]$ for some $m\in \Z$ and assume $0$ is a ghost vertex. Let $n\in \Z$, we define a weight function $w:\P [0,m]\to \Z_{\geq 0}$ to be $w(J)=\sum\limits_{j\in J}j$. The map $w$ induces the bounded decreasing filtration $(F^*_p)_{0\leq p\leq m(m+1)/2}$ on $CH_n^*(\ZZ_\K)$ given by
    \[F_p^l=\bigoplus_{\substack{J\subseteq [0,m]\\|J|=l\\w(J)\geq p}} \H_{n-1}(\K_J).\]
    As the differential in $CH_{n}^*(\ZZ_\K)$ never decreases the weight, this is a cochain complex filtration. The induced spectral sequence converges to the associated graded of $CH_{n}^*(\ZZ_\K)$ and has $E_0$ page given by
    \[E_0^{p,q}\cong \frac{F_p^{p+q}}{F_{p+1}^{p+q}}\cong \bigoplus_{\substack{J\subseteq [0,m]\\|J|=p+q\\w(J)=p}} \H_{n-1}(\K_J).\]
    As the only case where the differential does not move us to a higher filtration is adding the point $0$, the differential in the $E_0$ page $\partial_0^{p,q}: E_0^{p,q}\to E_0^{p,q+1}$ is given by
    \[\partial_0^{p,q}=(-1)^{n}\sum_{\substack{J\subseteq [m]\\|J|=p+q\\w(J)=p}} \varepsilon(0,J)\phi_{n-1;J,0}.\]
    However, as the inclusion $\K_J\xhookrightarrow{}\K_{J\cup\{0\}}$ is the identity, then each $\phi_{n-1;J,0}$ is an isomorphism, making $E_0$ acyclic. Therefore $E_\infty=E_1=0$, meaning that $HH_n^*(\ZZ_\K)=0$ for all $n\in \Z$ and so the result follows.
\end{proof}

\section{Main Result}
\begin{definition}
    We say an abstract simplicial complex $\K$ on $[m]$ is \textit{wedge-decomposable} if there are two simplicial complexes $\K^1$, $\K^2$ on $[m]$ such that $\K=\K^1\cup \K^2$ and $\K^1\cap\K^2=\gen{\sigma}$ where $\sigma$ is a possibly empty proper face of both $\K^1$ and $\K^2$. We denote this by $\K=\K^1\sqcup_\sigma \K^2$.
\end{definition}
It is worth noticing that in such a decomposition, both $\K^1$ and $\K^2$ are full subcomplexes of $\K$. The main result in this section is a generalization of the following theorem.

\begin{theorem}[{{\cite[Thm.~6.5]{docoho}}}]\label{theoremsimplex}
    Let $\K$ be a simplicial complex on $[m]$ with no ghost vertices, such that there exists a simplicial complex $\K'$ on $[m]$, $\tau\in \K$ and a possibly empty proper face $\sigma\subset \tau$ such that $\K=\K'\cup\gen{\tau}$ and $\K'\cap\gen{\tau}=\gen{\sigma}$. Then either $\K$ is a simplex or 
    \[HH_{-k,2l}(\ZZ_\K)\cong\left\{\begin{array}{cl}
         \Z&\text{ for }(k,l)=(0,0)\text{ or }(-1,4)\\ 
         0& \text{ else.} 
    \end{array}\right.\]
\end{theorem}

 Notice that in the previous theorem, if $\K$ is not a simplex, then $\K=\K'\sqcup_\sigma \gen{\tau}$ is wedge-decomposable. Also, it is important to note that not all simplicial complexes with double homology in bidegrees $(0,0)$ and $(-1,4)$ can be constructed by attaching a simplex to a simplicial complex.\\
\begin{eg}
    \normalfont {Consider the simplicial complex on $[4]$ given by $\K=\gen{\{1,2\},\{1,3\},\{2,3\},\{2,4\},\{3,4\}}$ and depicted in Figure 1.\\
    
    \begin{figure}[h]
        \centering
        \includegraphics[scale=0.25]{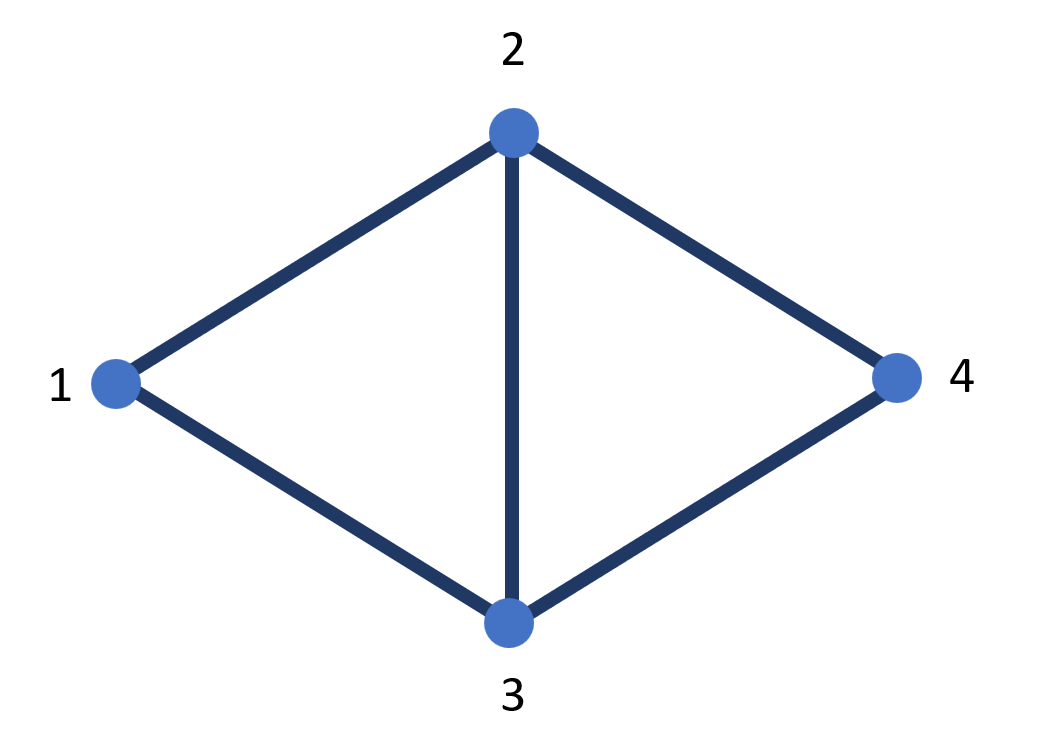}
        \caption{Geometric representation of $\K$.}
    \end{figure}
    \noindent Here the only non-zero reduced homology groups of full subcomplexes are
    \begin{align*}
        \H_{-1}(\K_\emptyset)\cong \Z\text{, }\H_{0}(\K_{\{1,4\}})\cong \Z \text{, }\H_1(\K_{\{1,2,3\}})\cong \Z\{1,2,3\}\text{, }\\
        \H_1(\K_{\{2,3,4\}})\cong \Z\{2,3,4\}\text{, and }\H_1(\K_{[m]})\cong\Z\{1,2,3\}\oplus\Z\{2,3,4\}.
    \end{align*}
    Therefore the homology of $\ZZ_\K$ is given by
    \begin{align*}
        H_{-k,2l}(\ZZ_\K)\cong \left\{\begin{array}{cl}
             \Z&\text{ for }(0,0) \text{ and }(-1,4)  \\
             \Z\{1,2,3\}\oplus\Z\{2,3,4\}& \text{ for }(-1,6) \text{ and }(-2,8)
        \end{array}\right.
    \end{align*}
    
    \noindent and so we automatically get that $HH_{0,0}(\ZZ_\K)\cong HH_{-1,4}(\ZZ_\K)\cong \Z$. The only nontrivial part of the cochain complex is 
    \[\begin{tikzcd}
	0 & {\Z\{1,2,3\}\oplus\Z\{2,3,4\}} & {\Z\{1,2,3\}\oplus\Z\{2,3,4\}} & 0
	\arrow[from=1-4, to=1-3]
	\arrow["d", from=1-3, to=1-2]
	\arrow[from=1-2, to=1-1]
\end{tikzcd}\]
    where $d$ is given by the matrix $\begin{bmatrix}
        -1 & 0 \\ 0 & 1
    \end{bmatrix}$ which is an isomorphism, and therefore $HH_{-1,6}(\ZZ_\K)=HH_{2,8}(\ZZ_\K)=0$. Here we have a simplicial complex with $HH$ of rank 2 concentrated in bidegrees $(0,0)$ and $(-1,4)$ that cannot be expressed as in Theorem \ref{theoremsimplex}. However, this is a wedge-decomposable simplicial complex as
    \[\K=\gen{\{1,2\},\{1,3\},\{2,3\}}\sqcup_{\{2,3\}}\gen{\{2,3\},\{2,4\},\{3,4\}}.\]
    }
\end{eg}
It turns out that wedge-decomposability is sufficient for the double homology to be concentrated in those degrees, see also \cite[ex. 6.7 and 6.8]{docoho}. To prove this we will use the following lemma.

\begin{lemma}\label{SES}
    Let $\K=\K^1\sqcup_\sigma \K^2$ be a wedge-decomposable simplicial complex on $[m]$ with no ghost vertices, if we denote by $\rho=V(\K^1)$, $\tau=V(\K^2)$  and let $\L=\gen{\rho}\sqcup_\sigma \gen{\tau}$, then for each non-empty $J\subseteq [m]$ and $n\geq 0$ there is a short exact sequence as follows
    \[\begin{tikzcd}
	0 & {\H_n(\K^1_J)\oplus\H_n(\K^2_J)} & {\H_n(\K_J)} & {\H_n(\L_J)} & 0.
	\arrow[from=1-2, to=1-3]
	\arrow[from=1-3, to=1-4]
	\arrow["{\phantom{}}", from=1-4, to=1-5]
	\arrow[from=1-1, to=1-2]
\end{tikzcd}\]
\end{lemma}
\begin{proof}
       Consider the sequence
			\begin{equation}\label{SES1}
		\begin{tikzcd}
			{\H_n(\gen{\sigma}_J)} & {\H_n(\K^1_J)\oplus\H_n(\K^2_J)} & {\H_n(\K_J)} & {\H_n(\L_J)} & 0
			\arrow["{\psi}", from=1-2, to=1-3]
			\arrow["\Phi", from=1-3, to=1-4]
			\arrow["{\phantom{}}", from=1-4, to=1-5]
			\arrow["{\varphi}", from=1-1, to=1-2]
		\end{tikzcd}\end{equation}
		  where $\varphi(\overline{z})=(\overline{z},-\overline{z})$,  $\psi(\overline{x},\overline{y})=\overline{x+y}$ and $\Phi$ is induced by the inclusion $\K_J\xhookrightarrow{}\L_J$ respectively; we claim sequence \eqref{SES1} is exact.\\

        For $n\geq 1$ this sequence is just the Mayer-Vietoris sequence for the pair $\K^1_J$ and $\K^2_J$  as $\H_n(\gen{\sigma}_J)=\H_n(\L_J)=0$, so we can assume $n=0$. Notice that Sequence \ref{SES1} is always exact at $\H_0(\K^1_J)\oplus \H_0(\K^2_J)$ as 
        \[(x,y)\in\ker{\psi}\iff \overline{x}=-\overline{y}\iff (\overline{x},\overline{y})=(\overline{x},-\overline{x})=\varphi(\overline{x})\in \text{Im }\varphi.\]

	If $J\cap \sigma\neq\emptyset$, then $\L_J$ has only one path-connected component and therefore $\H_0(\L_J)=0$. Using the Mayer-Vietoris exact sequence of the pair $\K^1_J$, $\K^2_J$, as $\K^1_J\cap\K^2_J=\gen{\sigma}_J\neq\emptyset$ then $\H_{-1}(\gen{\sigma}_J)=0$ and so we have the exact sequence
    \[\begin{tikzcd}
	   {\H_0(\K^1_J)\oplus\H_0(\K^2_J)} & {\H_0(\K_J)} & {\H_{-1}(\gen{\sigma}_J)}. 
	   \arrow["\psi", from=1-1, to=1-2]
	   \arrow["0", from=1-2, to=1-3]
    \end{tikzcd}\]
       Therefore $\psi$ is surjective, verifying exactness in this case.\\
        
		Now suppose $J\cap \sigma=\emptyset$, which implies $\gen{\sigma}_J=\emptyset$. We have three cases:
    \begin{itemize}
        \item ($J\subseteq \rho$) In this case $\K_J=\K^1_J$, $\L_J=\gen{\rho}_J$, and $\K^2_J=\emptyset$, meaning that $\H_0(\gen{\sigma}_J)=\H_0(\gen{\rho}_J)=\H_0(\K^2_J)=\H_0(\emptyset)=0$ and therefore the sequence (\ref{SES1}) becomes
        \[\begin{tikzcd}
	   {\H_0(\gen{\sigma}_J)} & {\H_0(\K^1_J)}\oplus 0 & {\H_0(\K^1_J)} & {\H_0(\gen{\rho}_J)} & 0
	   \arrow["0", from=1-1, to=1-2]
	   \arrow["1", from=1-2, to=1-3]
	   \arrow["0", from=1-3, to=1-4]
	   \arrow[from=1-4, to=1-5]
    \end{tikzcd}\]
which is exact. 
        \item ($J\subseteq \tau$) This case is analogous to the previous one switching the places of $\rho$ and $\tau$.
        \item ($J\cap\rho \neq \emptyset\neq J\cap \tau$) Let $n_1$ be the number of path components of $\K^1_J$. For each path component $V_i$ of $\K^1_J$ we choose a vertex $x_i\in V_i$, obtaining this way a basis for $H_0(\K^1_J)$, namely $\{\overline{x}_i\}_{i\in [n_1]}$. Through an analogous process, if $n_2$ is the number of path components of $\K^2_J$ we have the basis $\{\overline{y}_i\}_{i\in [n_2]}$ for $H_0(\K^2_J)$. Similarly, we take $x\in \gen{\rho}_J$ and $y\in \gen{\tau}_J$, as $\L_J$ has only two path components and both $x$ and $y$ correspond to different ones we obtain that $\{\overline{x},\overline{y}\}$ is a basis for $H_0(\L_J)$. This way we can present the reduced homology groups as
            \begin{align*}
			\H_0(\K^1_J)&\cong \gen{\overline{x_i-x_1}: i\in [2,n_1] }\\
			\H_0(\K^2_J)&\cong \gen{\overline{y_i-y_1}:i\in [2,n_2]}\\
			\H_0(\K_J)&\cong \gen{\overline{x_i-x_1},\overline{y_j-x_1}:\; i\in [2,n_1],\;j\in [n_2] }\\
			\H_0(\L_J)&\cong \gen{\overline{y-x}}.
		\end{align*}
  Therefore we have that \begin{align*}
			\Phi(\overline{x_i-x_1})=0\\
			\Phi(\overline{y_i-x_1})=\overline{y-x},
		\end{align*}
		and so $\Phi$ is surjective with kernel
            \[\text{Ker } (\Phi)=\gen{\overline{x_i-x_1}, \left(\overline{y_j-x_1}\right)-\left(\overline{y_1-x_1}\right):i\in [2,n_1],\;j\in [2,n_2]},\]
            meaning that the sequence is exact at $\H_0(\L_J)$.\\
        \noindent On the other hand, we have that 
		\begin{align*}
			\psi(\overline{x_i-x_1})&=\overline{x_i-x_1}\\
			\psi(\overline{y_j-y_1})&=\left(\overline{y_j-x_1}\right)-\left(\overline{y_1-x_1}\right)
		\end{align*}
    	therefore \[\text{Im } (\psi)=\gen{\overline{x_i-x_1}, \left(\overline{y_j-x_1}\right)-\left(\overline{y_1-x_1}\right):i\in [2,n_1],\;j\in [2,n_2]},\]
        meaning that the sequence is exact at $\H_0(\K_J)$.\\
        As $\gen{\sigma}$ is a simplex, so is $\gen{\sigma}_J$ and therefore $\H_0(\gen{\sigma}_J)=0$ completing the proof.\qedhere
    \end{itemize}
\end{proof}

\begin{theorem}\label{main}
	Let $\K$ be a wedge decomposable simplicial complex on $[m]$ with no ghost vertices, then 
\[
	HH_{n}^l(\ZZ_\K)\cong\left\{\begin{array}{cl}
		\Z & \text{ for $(n,l)=(0,0)$ and $(1,2)$}\\
		0 & \text{ else.}
	\end{array}\right.\]
    Or equivalently
    \[HH_{-k,2l}(\ZZ_\K)\cong\left\{\begin{array}{cl}
		\Z & \text{ for $(-k,2l)=(0,0)$ and $(-1,4)$}\\
		0 & \text{ else.}
	\end{array}\right. \]
\end{theorem}
\begin{proof}
Let $\K=\K^1\sqcup_\sigma \K^2$ and assume $\K$ has no ghost vertices. Proposition \ref{boo} proves the case when $l=0$, so we shall assume $l>0$. As $\K$ has no ghost vertices we can assume $n>0$.\\

Let $J\subseteq[m]$ where $|J|=l$, we define $\L=\gen{V(\K^1)}\sqcup_\sigma\gen{V(\K^2)}$. Notice that $\L_{[m]\setminus\sigma}$ has two connected components, then $\L$ is not a simplex. By Lemma \ref{SES}, there is a short exact sequence
            \begin{equation}\label{ses}
                \begin{tikzcd}
	0 & {\H_{n-1}(\K^1_J)\oplus\H_{n-1}(\K^2_J)} & {\H_{n-1}(\K_J)} & {\H_{n-1}(\L_J)} & 0.
	\arrow[from=1-2, to=1-3]
	\arrow[from=1-3, to=1-4]
	\arrow[from=1-4, to=1-5]
	\arrow[from=1-1, to=1-2]
\end{tikzcd}
            \end{equation}    
    We will show these exact sequences induce a short exact sequence at the level of $CH_n^*$. Let $x\in [m]\setminus J$, consider the commutative diagrams
    \[\begin{tikzcd}
	{\K^1_J} & {\K_J} & {\L_J} \\
	{\K^1_{J\cup\{x\}}} & {\K_{J\cup\{x\}}} & {\L_{J\cup\{x\}}}
	\arrow["i",hook, from=1-1, to=1-2]
	\arrow["h",hook, from=1-2, to=1-3]
	\arrow[hook, from=1-1, to=2-1]
	\arrow["i",hook, from=2-1, to=2-2]
	\arrow["h",hook, from=2-2, to=2-3]
	\arrow[hook, from=1-3, to=2-3]
	\arrow[hook, from=1-2, to=2-2]
\end{tikzcd}\begin{tikzcd}
	{\K^2_J} & {\K_J} & {\L_J} \\
	{\K^2_{J\cup\{x\}}} & {\K_{J\cup\{x\}}} & {\L_{J\cup\{x\}}}
	\arrow["j", hook, from=1-1, to=1-2]
	\arrow["h", hook, from=1-2, to=1-3]
	\arrow[hook, from=1-1, to=2-1]
	\arrow["j",hook, from=2-1, to=2-2]
	\arrow["h", hook, from=2-2, to=2-3]
	\arrow[hook, from=1-3, to=2-3]
	\arrow[hook, from=1-2, to=2-2]
\end{tikzcd}\]
with all the arrows being the canonical inclusions. As $\H$ is functorial with respect to inclusions, then for each $n\geq 0$ we get the following commutative diagrams
\[\begin{tikzcd}
	{\H_{n-1}(\K^1_J)} & {\H_{n-1}(\K_J)} & {\H_{n-1}(\L_J)} \\
	{\H_{n-1}(\K^1_{J\cup\{x\}})} & {\H_{n-1}(\K_{J\cup\{x\}})} & {\H_{n-1}(\L_{J\cup\{x\}}),}
	\arrow["{i_*}", from=1-1, to=1-2]
	\arrow["{h_*}", from=1-2, to=1-3]
	\arrow["{\phi^{\K^1}_{n-1;J,x}}"', from=1-1, to=2-1]
	\arrow["{i_*}"',  from=2-1, to=2-2]
	\arrow["{h_*}"',  from=2-2, to=2-3]
	\arrow["{\phi^{\L}_{n-1;J,x}}", from=1-3, to=2-3]
	\arrow["{\phi^{\K}_{n-1;J,x}}"', from=1-2, to=2-2]
\end{tikzcd}\]
    \[\begin{tikzcd}
	{\H_{n-1}(\K^2_J)} & {\H_{n-1}(\K_J)} & {\H_{n-1}(\L_J)} \\
	{\H_{n-1}(\K^2_{J\cup\{x\}})} & {\H_{n-1}(\K_{J\cup\{x\}})} & {\H_{n-1}(\L_{J\cup\{x\}}).}
	\arrow["{j_*}", from=1-1, to=1-2]
	\arrow["{h_*}", from=1-2, to=1-3]
	\arrow["{\phi^{\K^2}_{n-1;J,x}}"', from=1-1, to=2-1]
	\arrow["{j_*}"',  from=2-1, to=2-2]
	\arrow["{h_*}"',  from=2-2, to=2-3]
	\arrow["{\phi^{\L}_{n-1;J,x}}", from=1-3, to=2-3]
	\arrow["{\phi^{\K}_{n-1;J,x}}"', from=1-2, to=2-2]
\end{tikzcd}\]
 As $\oplus$ is a biproduct in $\Z$\textbf{-mod} we can combine these diagrams, and while we are at it, we can multiply by the sign $(-1)^{n}\varepsilon(x,J)$ all columns without affecting the commutativity. This way we obtain the commutative diagram
\[\hspace{-5mm}\begin{tikzcd}
	& {\H_{n-1}(\K^1_J)\oplus\H_{n-1}(\K^2_J)} & {\H_{n-1}(\K_J)} & {\H_{n-1}(\L_J)}& \\
	& {\H_{n-1}(\K^1_{J\cup\{x\}})}\oplus {\H_{n-1}(\K^2_{J\cup\{x\}})} & {\H_{n-1}(\K_{J\cup\{x\}})} & {\H_{n-1}(\L_{J\cup\{x\}}).}&
	\arrow["{i_*+j_*}", from=1-2, to=1-3]
	\arrow["{h_*}", from=1-3, to=1-4]
	\arrow["{(-1)^{n}\varepsilon(x,J)\phi^{\K^1}_{n-1;J,x}\oplus\phi^{\K^2}_{n-1;J,x}}"', from=1-2, to=2-2]
	\arrow["{i_*+j_*}"',  from=2-2, to=2-3]
	\arrow["{h_*}"',  from=2-3, to=2-4]
	\arrow["{(-1)^{n}\varepsilon(x,J)\phi^{\L}_{n-1;J,x}}", from=1-4, to=2-4]
	\arrow["{(-1)^{n}\varepsilon(x,J)\phi^{\K}_{n-1;J,x}}"', from=1-3, to=2-3]
\end{tikzcd}\]
    As $\oplus$ is a product, we take direct sum over all $x\in [m]\setminus J$. Also, as the horizontal maps are precisely the same as those in Sequence \eqref{SES1} we get the following commutative diagram with exact rows
    \[\adjustbox{scale=0.9}{\begin{tikzcd}
	0& {\H_{n-1}(\K^1_J)\oplus\H_{n-1}(\K^2_J)} & {\H_{n-1}(\K_J)} & {\H_{n-1}(\L_J)}&0 \\
	0& \bigoplus\limits_{\substack{x\in [m]\setminus J}}\H_{n-1}(\K^1_{J\cup\{x\}}) \oplus \H_{n-1}(\K^2_{J\cup\{x\}})& \bigoplus\limits_{\substack{x\in [m]\setminus J}}{\H_{n-1}(\K_{J\cup\{x\}})} & \bigoplus\limits_{\substack{x\in [m]\setminus J}}{\H_{n-1}(\L_{J\cup\{x\}})}&0
    \arrow[from=1-1,to=1-2]
	\arrow["{\psi}", from=1-2, to=1-3]
	\arrow["{\Phi}", from=1-3, to=1-4]
    \arrow[from=1-4,to=1-5]
	\arrow["d_{n-1;J}\oplus d_{n-1;J}"', from=1-2, to=2-2]
    \arrow[from=2-1,to=2-2]
	\arrow["{\psi}"',  from=2-2, to=2-3]
	\arrow["{\Phi}"',  from=2-3, to=2-4]
    \arrow[from=2-4,to=2-5]
	\arrow["d_{n-1;J}"', from=1-4, to=2-4]
	\arrow["d_{n-1;J}"', from=1-3, to=2-3]
\end{tikzcd}}\]
    where the vertical maps are defined as in Equation \eqref{partiald}. We can think of the bottom row as taking the direct sum indexed over all subsets of $[m]$ of size $l+1$ that contain $J$; we can extend this sum by inclusion to all subsets of size $l+1$ while maintaining both commutativity and exactness:
    \[\adjustbox{scale=1}{\begin{tikzcd}
	0& {\H_{n-1}(\K^1_J)\oplus\H_{n-1}(\K^2_J)} & {\H_{n-1}(\K_J)} & {\H_{n-1}(\L_J)}&0 \\
	0& \bigoplus\limits_{\substack{I\subseteq [m]\\|I|=l+1}}\H_{n-1}(\K^1_{I})\oplus \bigoplus\limits_{\substack{I\subseteq [m]\\|I|=l+1}}\H_{n-1}(\K^2_{I}) & \bigoplus\limits_{\substack{I\subseteq [m]\\|I|=l+1}}{\H_{n-1}(\K_{I})} & \bigoplus\limits_{\substack{I\subseteq [m]\\|I|=l+1}}{\H_{n-1}(\L_{I})}&0.
    \arrow[from=1-1,to=1-2]
	\arrow["{\psi}", from=1-2, to=1-3]
	\arrow["{\Phi}", from=1-3, to=1-4]
    \arrow[from=1-4,to=1-5]
	\arrow["d_{n-1;J}\oplus d_{n-1;J}"', from=1-2, to=2-2]
    \arrow[from=2-1,to=2-2]
	\arrow["{\psi}"',  from=2-2, to=2-3]
	\arrow["{\Phi}"',  from=2-3, to=2-4]
    \arrow[from=2-4,to=2-5]
	\arrow["d_{n-1;J}"', from=1-4, to=2-4]
	\arrow["d_{n-1;J}"', from=1-3, to=2-3]
\end{tikzcd}}\]
    As $\oplus$ is a coproduct in $\Z$\textbf{-mod}, we can take the direct sum over all possible $J$. Therefore from the definition of $CH_n^*(\ZZ_\K)$ we obtain the following commutative diagram with exact rows
 \[\adjustbox{scale=1}{\begin{tikzcd}
	0& CH_n^l(\ZZ_{\K^1})\oplus CH_n^l(\ZZ_{\K^2}) & CH_n^l(\ZZ_{\K}) & CH_n^l(\ZZ_{\L})&0 \\
	0& CH_n^{l+1}(\ZZ_{\K^1})\oplus CH_n^{l+1}(\ZZ_{\K^2}) & CH_n^{l+1}(\ZZ_{\K}) & CH_n^{l+1}(\ZZ_{\L})&0.
    \arrow[from=1-1,to=1-2]
	\arrow["{\psi}", from=1-2, to=1-3]
	\arrow["{\Phi}", from=1-3, to=1-4]
    \arrow[from=1-4,to=1-5]
	\arrow[""', from=1-2, to=2-2]
    \arrow[from=2-1,to=2-2]
	\arrow["{\psi}"',  from=2-2, to=2-3]
	\arrow["{\Phi}"',  from=2-3, to=2-4]
    \arrow[from=2-4,to=2-5]
	\arrow["", from=1-4, to=2-4]
	\arrow[""', from=1-3, to=2-3]
\end{tikzcd}}\]
    
\noindent As this holds for every $l>0$ this induces a short exact sequence of truncated cochain complexes
            \[\begin{tikzcd}
	0 & {CH_n^{*}(\ZZ_{\K^1})\oplus CH_n^{*}(\ZZ_{\K^2})} & {CH_n^{*}(\ZZ_{\K})} & {CH_n^{*}(\ZZ_{\L})} & 0.
	\arrow[from=1-2, to=1-3]
	\arrow[from=1-1, to=1-2]
	\arrow[from=1-3, to=1-4]
	\arrow[from=1-4, to=1-5]
\end{tikzcd}\]
Therefore, for $l>0$ we get the following exact sequence in cohomology
\begin{equation} \label{finalles}
    \adjustbox{scale=0.9}{\begin{tikzcd}
	{HH_n^l(\ZZ_{\K^1})\oplus HH_n^l(\ZZ_{\K^2})} & {HH_n^l(\ZZ_{\K})} & {HH_n^{l+1}(\ZZ_{\L})} & HH_n^{l+1}(\ZZ_{\K^1})\oplus HH_n^{l+1}(\ZZ_{\K^2}).
	\arrow[from=1-1, to=1-2]
	\arrow[from=1-2, to=1-3]
	\arrow[from=1-3, to=1-4]
\end{tikzcd}}
\end{equation}
However, as $\sigma$ is a proper face of both $\K^1$ and $\K^2$ and each of them are non-empty, there exist $x,y\in [m]$ such that $\{x\}\notin \K^1$ and $\{y\}\notin \K^2$; that is, both $\K^1$ and $\K^2$ have ghost vertices. Applying Theorem \ref{king boo}, \[HH_{*,*}(\ZZ_{\K^1})=HH_{*,*}(\ZZ_{\K^2})=0.\] Therefore the map $HH_n^l(\ZZ_\K)\to HH_n^{l}(\ZZ_\L)$ in Sequence \eqref{finalles} is an isomorphism. Finally, as $\L$ satisfies the hypotheses in Theorem \ref{theoremsimplex} it follows that
\begin{align*}
    HH_n^l(\ZZ_\K)\cong HH_n^l(\ZZ_\L)\cong HH_{-(l-n),2l}(\ZZ_\L)&\cong \left\{\begin{array}{cl}
    \Z  & \text{for $(-(l-n),2l)=(0,0)$ or $(-1,4)$} \\
     0  & \text{else}
\end{array}\right.\\
&\cong \left\{\begin{array}{cl}
    \Z  & \text{for $(n,l)=(0,0)$ or $(1,2)$} \\
     0  & \text{else,}
\end{array}\right.
\end{align*}
completing the proof.
\end{proof}
\begin{eg}\label{counter}
    \normalfont{Consider the simplicial complex on $[5]$ given by $\L=\gen{\{1,2,3\},\{1,3,4\},\{1,4,5\},\{2,3,4\},\{2,4,5\}}$ depicted in Figure 2 below.\\
    \begin{figure}[h]
        \centering
        \includegraphics[scale=0.35]{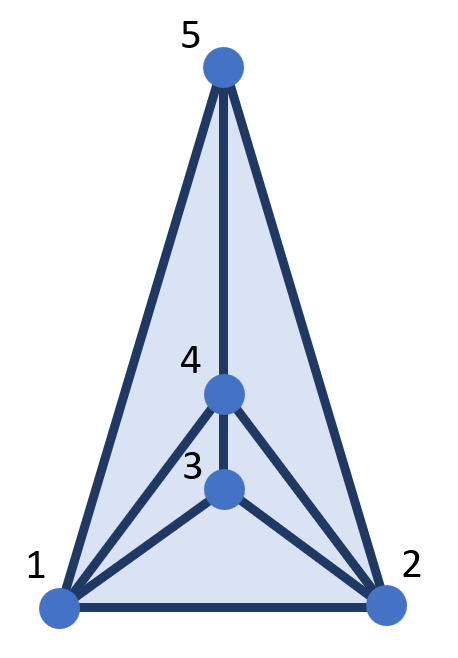}
        \caption{Geometric representation of $\L$.}
    \end{figure}

\noindent The only nontrivial reduced homology groups of full subcomplexes are
    \begin{align*}
        &\H_{-1}\left(\L_\emptyset\right)\cong\Z\text{, }&\H_0\left(\L_{\{3,5\}}\right)\cong\Z&,\\
        &\H_{1}\left(\L_{\{1,2,4\}}\right)=\Z\{1,2,4\}\text{, }&\H_1\left(\L_{\{1,2,5\}}\right)=\Z\{1,2,5\}&,\\
        &\H_{1}\left(\L_{\{1,2,4,5\}}\right)=\Z\{1,2,4,5\}\text{, }&\H_{1}\left(\L_{\{1,2,3,5\}}\right)=\Z\{1,2,3,5\}&.
    \end{align*}
    Therefore, from Hochster's formula
    \[H_{-k,2l}(\ZZ_\L)\cong\left\{\begin{array}{cl}
        \Z & \text{for $(0,0)$ and $(-1,4)$}, \\
        \Z\{1,2,4\} \oplus\Z\{1,2,5\}& \text{for $(-1, 6)$},\\
        \Z\{1,2,3,4\}\oplus \Z\{1,2,4,5\}& \text{for $(-2, 8)$},\\
        0&\text{else.}
    \end{array}\right.\]
    And so we automatically get that $HH_{0,0}(\ZZ_\L)\cong HH_{-1,4}(\ZZ_\L)\cong\Z$. The only nontrivial part of the cochain complex is 
     \[\begin{tikzcd}
	0 & {\Z\{1,2,3,4\}\oplus\Z\{1,2,4,5\}} & {\Z\{1,2,4\}\oplus\Z\{1,2,5\}} & 0
	\arrow[from=1-4, to=1-3]
	\arrow["d", from=1-3, to=1-2]
	\arrow[from=1-2, to=1-1]
\end{tikzcd}\]
where $d$ is given by the matrix $\begin{bmatrix}
    -1&1\\
    0&1
\end{bmatrix}$, which is an isomorphism. So we can conclude $HH_{-k,2l}(\ZZ_\L)$ is as in Theorem \ref{main}. However, $\L$ is not wedge decomposable. To see this, notice that for a wedge-decomposable complex $\K^1\sqcup_\sigma\K^2$, the vertices in $\K^1\setminus\gen{\sigma}$ cannot be adjacent to those in $\K^2\setminus\gen{\sigma}$; so if we could write $\L=\K^1\sqcup_\sigma\K^2$, since $1,2$ and $4$ are adjacent to all vertices, they would have to be in $\sigma$, but $\{1,2,4\}$ is not a simplex of $\L$.}
\end{eg}

\bibliographystyle{alpha}
\bibliography{ref.bib}

\newcommand{\etalchar}[1]{$^{#1}$}
\begin{thebibliography}{BLP{\etalchar{+}}23}

\bibitem[BLP{\etalchar{+}}23]{Stab}
Anthony Bahri, Ivan Limonchenko, Taras Panov, Jongbaek Song, and Donald
  Stanley.
\newblock A stability theorem for bigraded persistence barcodes.
\newblock 2023.
\newblock arXiv 2303.14694.

\bibitem[BP14]{ToricTop}
Victor Buchstaber and Taras Panov.
\newblock {\em Toric Topology}.
\newblock 2014.

\bibitem[Hat00]{Hatcher}
Allen Hatcher.
\newblock {\em {Algebraic topology}}.
\newblock Cambridge Univ. Press, Cambridge, 2000.

\bibitem[Hoc77]{hochster}
M.~Hochster.
\newblock Cohen-macaulay rings, combinatorics, and simplicial complexes, in
  ``ring theory ii".
\newblock {\em Lect. Notes in Pure Appl. Math.}, (26):171--223, 1977.

\bibitem[LPSS23]{docoho}
Ivan Limonchenko, Taras Panov, Jongbaek Song, and Donald Stanley.
\newblock Double cohomology of moment-angle complexes.
\newblock {\em Advances in Mathematics}, 432:109274, 2023.

\bibitem[Spa66]{Spanier}
Edwin~H. Spanier.
\newblock {\em Algebraic Topology}.
\newblock Springer New York, 1966.

\end{thebibliography}
\end{document}